\documentclass[10pt]{amsart}

\usepackage[margin=1in]{geometry}
\usepackage{amsmath}
\usepackage{float}
\usepackage{booktabs}
\usepackage[dvipsnames]{xcolor} 

\usepackage{fix-cm}

\usepackage{ wasysym }

\usepackage{blkarray}

\usepackage{amscd,amsmath,amssymb,amsfonts,amsthm}
\usepackage{enumerate,mathrsfs,stmaryrd,latexsym, comment, mathtools, mathdots} 

\usepackage[mathscr]{euscript}

\usepackage{caption}
\usepackage{subcaption}
\usepackage{hyperref}

\usepackage[all]{xy}

\usepackage{tikz-cd}
\usepackage{tikz}
\usetikzlibrary{snakes, 
	3d, matrix,decorations.pathreplacing,calc,arrows,decorations.pathmorphing, patterns}
\usetikzlibrary {positioning}
\usetikzlibrary{calc,backgrounds}

\pgfdeclarelayer{background}
\pgfdeclarelayer{foreground}
\pgfsetlayers{background,main,foreground}

\usepackage{tikz-3dplot}

\numberwithin{equation}{section}
\allowdisplaybreaks[1]

\makeatletter
\newcommand{\leqnomode}{\tagsleft@true\let\veqno\@@leqno}
\newcommand{\reqnomode}{\tagsleft@false\let\veqno\@@eqno}
\makeatother

\def\ke{K\"ahler--Einstein }

\DeclareMathOperator{\SL}{SL}

\DeclareMathOperator{\Gr}{Gr}
\DeclareMathOperator{\Cone}{Cone}

\DeclareMathOperator{\End}{End}

\DeclareMathOperator{\Hom}{Hom}
\DeclareMathOperator{\Sym}{Sym}

\DeclareMathOperator{\Sp}{Sp}

\DeclareMathOperator{\SO}{SO}
\DeclareMathOperator{\LGr}{LGr}

\DeclareMathOperator{\PGL}{PGL}
\DeclareMathOperator{\Pic}{Pic}

\newtheorem{theorem}{Theorem}[section]

\newtheorem{proposition}[theorem]{Proposition}
\newtheorem{corollary}[theorem]{Corollary}

\theoremstyle{definition}
\newtheorem{example}[theorem]{Example}
\newtheorem{definition}[theorem]{Definition}
\newtheorem{remark}[theorem]{Remark}

\newtheorem*{theorem1}{Theorem 1.1} 

\begin{document}
	
\title[K\"{a}hler--Einstein metrics on smooth Fano symmetric varieties]{K\"{a}hler--Einstein metrics on \\ smooth Fano symmetric varieties with Picard number one}

\author{Jae-Hyouk Lee}
\address{Department of Mathematics, Ewha Womans University, 
Seodaemun-gu, Seoul 03760, Korea}
\email{jaehyoukl@ewha.ac.kr}

\author{Kyeong-Dong Park}
\address{School of Mathematics, Korea Institute for Advanced Study (KIAS), Dongdaemun-gu, Seoul 02455, Korea}
\email{kdpark@kias.re.kr}

\author{Sungmin Yoo}
\address{Center for Geometry and Physics, Institute for Basic Science (IBS), Pohang 37673, Korea}
\email{sungmin@ibs.re.kr}



\begin{abstract}
Symmetric varieties are normal equivarient open embeddings of symmetric homogeneous spaces, and they are interesting examples of spherical varieties. 
We prove that all smooth Fano symmetric varieties with Picard number one admit K\"{a}hler--Einstein metrics by using a combinatorial criterion for K-stability of Fano spherical varieties obtained by Delcroix. 
For this purpose, we present their algebraic moment polytopes and compute the barycenter of each moment polytope with respect to the Duistermaat--Heckman measure. 
\end{abstract}

\maketitle
\setcounter{tocdepth}{1} 
\date{\today}

\section{Introduction}

A K\"{a}hler metric on a complex manifold is said to be \emph{K\"{a}hler--Einstein} if the Riemannian metric defined by its real part has constant Ricci curvature. 
The existence of \ke metrics on Fano manifolds has become a central topic in complex geometry in recent years.
In contrast to Calabi--Yau and general type (\cite{aubin78, yau78}), Fano manifolds do not necessarily have a \ke metric in general, 
and there are obstructions based on the (holomorphic) automorphism group.

The first obstruction was discovered by Matsushima in \cite{matsushima}. 
He proved that the reductivity of the automorphism group is a necessary condition for the existence of \ke metrics.
Later, Futaki \cite{futaki} proved that the existence of \ke metrics implies that the Futaki invariant, a functional on the Lie algebra of the automorphism group, vanishes.
As a generalization of this invariant on test configurations, Tian \cite{tian94, tian97} and Donalson \cite{donaldson02} introduced a certain algebraic stability condition, which is called the \emph{K-stability}.
The famous Yau--Tian--Donaldson conjecture predicts that the existence of a \ke metric on a Fano manifold is equivalent to the K-stability.
Eventually, this conjecture was solved by Chen--Donaldson--Sun \cite{cds1, cds2, cds3} and Tian \cite{tian15}. 

Despite of these obstructions, each Fano homogeneous manifold admits a \ke metric \cite{mat2, koszul}.
Therefore, one can expect the existence of a \ke metric on a Fano manifold if it has large automorphism group.
A natural candidate is the \emph{almost-homogeneous} manifold, that is, a manifold 
with an open dense orbit of a complex Lie group. 
For the case of toric Fano manifolds, Wang and Zhu \cite{wz} proved that 
the existence of a K\"{a}hler--Einstein metric is equivalent to the vanishing of the Futaki invariant. 
In fact, 
this was based on the theorem by Mabuchi \cite{mabuchi}, which says that the Futaki invariant vanishes if and only if the barycenter of the moment polytope is 
the origin.
This gave us a powerful combinatorial criterion for the existence of a \ke metric on a toric Fano manifold, which is much easier to check than the K-stability condition.

An important class of almost-homogeneous varieties is \emph{spherical} varieties 
including toric varieties, \emph{group compactifications} (\cite{Del17}), and \emph{symmetric} varieties.
A normal variety is called \emph{spherical} if it admits an action of a reductive group of which a Borel subgroup acts with an open orbit on the variety.
As a generalization of Wang and Zhu's work, 
Delcroix \cite{Del16} extended a combinatorial criterion for 
K-stability of Fano spherical manifolds, 
in terms of its moment polytope and spherical data. 
In particular, this criterion is also applicable to smooth Fano symmetric varieties (see Corollary 5.9 of \cite{Del16}).

By combining the above criterion and Ruzzi's classification \cite{Ruzzi2011} of smooth Fano symmetric varieties with Picard number one, 
we prove the following.

\begin{theorem}\label{Main theorem} 
All smooth Fano symmetric varieties with Picard number one admit \ke metrics. 
\end{theorem}

For this theorem, the condition on the Picard number is crucial 
because a smooth Fano symmetric variety with higher Picard number may have no \ke metrics. 
For example, the blow-up of the wonderful compactification of $\Sp(4, \mathbb C)$ along the closed orbit does not admit any \ke metrics
(see Example 5.4 of \cite{Del17}). 
Moreover, we 
note that Delcroix already provided the existence of \ke metrics 
on smooth Fano embedding of $\SL(3,\mathbb{C})/\SO(3,\mathbb{C})$,
and group compactifications of $\SL(3,\mathbb{C})$ and $G_{2}$ respectively
(see Example 5.13 of \cite{Del16}). 
The above theorem leads us to complete all remaining cases of smooth Fano symmetric varieties with Picard number one also admit \ke metrics.

\vskip 1em

\noindent
\textbf{Acknowledgements}. 
The second author would like to express thanks to Jihun Park and Eunjeong Lee for their interests and helpful comments. 

The first author was supported by the National Research Foundation of Korea (NRF) grant funded by the Korea government (MSIT)
(NRF-2019R1F1A1058962). 
The second author was supported by the Institute for Basic Science (IBS-R003-D1), 
and by the National Research Foundation of
Korea (NRF) grant funded by the Korea government (MSIT)
(NRF-2019R1A2C3010487). 
The third author was supported by the Institute for Basic Science (IBS-R003-D1). 

\section{Criterion for existence of K\"{a}hler--Einstein metrics on symmetric varieties} 

Let $G$ be a connected reductive algebraic group over $\mathbb C$. 

\subsection{Spherical varieties and algebraic moment polytopes} 
We review general notions and results about spherical varieties. 
The normal equivariant embeddings of a given spherical homogeneous space are classified by combinatorial objects called \emph{colored fans}, 
which generalize the fans appearing in the classification of toric varieties. 
We refer \cite{Knop91}, \cite{Timashev11} and \cite{Gandini18} as references for spherical varieties. 

\begin{definition}
\label{spherical variety}
A normal variety $X$ equipped with an action of $G$ is called \emph{spherical} if a Borel subgroup $B$ of $G$ acts on $X$ with an open and dense orbit.
\end{definition}

Let $G/H$ be an open dense $G$-orbit of a spherical variety $X$ and $T$ a maximal torus of $B$. 
By definition, the \emph{spherical weight lattice} $\mathcal M$ of $G/H$ is 
a subgroup of characters $\chi \in \mathfrak X(B) = \mathfrak X(T)$ of (nonzero) $B$-semi-invariant functions in the function field $\mathbb C(G/H) = \mathbb C(X)$, 
that is, $$\mathcal M = \{ \chi \in \mathfrak X(T) : \mathbb C(G/H)^{(B)}_{\chi} \neq 0 \},$$ 
where $\mathbb C(G/H)^{(B)}_{\chi} = \{ f \in \mathbb C(G/H) : b \cdot f = \chi(b) f \text{ for all } b \in B \}$. 
Note that every function $f_{\chi}$ in $\mathbb C(G/H)^{(B)}$ is determined by its weight $\chi$ up to constant 
because $\mathbb C(G/H)^{B} = \mathbb C$, that is, any $B$-invariant rational function on $X$ is constant. 
The spherical weight lattice $\mathcal M$ is a free abelian group of finite rank. 
We define the \emph{rank} of $G/H$ as the rank of the lattice $\mathcal M$. 
Let $\mathcal N = \Hom(\mathcal M, \mathbb Z)$ denote its dual lattice together with the natural pairing $\langle \, \cdot \, , \, \cdot \, \rangle \colon \mathcal N \times \mathcal M \to \mathbb Z$.

As the open $B$-orbit of a spherical variety $X$ is an affine variety, 
its complement has pure codimension one and is a finite union of $B$-stable prime divisors. 

\begin{definition}
\label{color}
For a spherical variety $X$, 
$B$-stable but not $G$-stable prime divisors in $X$ are called \emph{colors} of $X$. 
A color of $X$ corresponds to a $B$-stable prime divisor in the open $G$-orbit $G/H$ of $X$. 
We denote by $\mathfrak D = \{ D_1, \cdots, D_k \}$ the set of colors of $X$ (or $G/H$).  
\end{definition}

As a $B$-semi-invariant function $f_{\chi}$ in $\mathbb C(G/H)^{(B)}_{\chi}$ is unique up to constant, 
we define the \emph{color map} $\rho \colon \mathfrak D \to \mathcal N$ by $\langle \rho(D), \chi \rangle = \nu_D(f_{\chi})$ for $\chi \in \mathcal M$, 
where $\nu_D$ is the discrete valuation associated to a divisor $D$, that is, $\nu_D(f)$ is the vanishing order of $f$ along $D$. 
Unfortunately, the color map is generally not injective.
In addition, every discrete $\mathbb Q$-valued valuation $\nu$ of the function field $\mathbb C(G/H)$ induces a homomorphism $\hat{\rho}(\nu) \colon \mathcal M \to \mathbb Q$ defined by $\langle \hat{\rho}(\nu), \chi \rangle = \nu(f_{\chi})$, so that we get a map $\hat{\rho} \colon \{ \text{discrete $\mathbb Q$-valued valuations on $G/H$} \} \to \mathcal N \otimes \mathbb Q$. 
Luna and Vust \cite{LV83} showed that the restriction of $\hat{\rho}$ to the set of $G$-invariant discrete valuations on $G/H$ is injective. 
From now on, we will regard a $G$-invariant discrete valuation on $G/H$ as an element of $\mathcal N \otimes \mathbb Q$ via the map $\hat{\rho}$, 
and in order to simplify the notation $\hat{\rho}(\nu_E)$ will be written as $\hat{\rho}(E)$ for a $G$-stable divisor $E$ in $X$. 

Let $L$ be a $G$-linearized ample line bundle on a spherical $G$-variety $X$. 
By the multiplicity-free property of spherical varieties, 
the algebraic moment polytope $\Delta(X, L)$ encodes the structure of representation of $G$ in the spaces of multi-sections of tensor powers of $L$. 

\begin{definition}
\label{moment polytope}
The \emph{algebraic moment polytope} $\Delta(X, L)$ of $L$ with respect to $B$ is defined as 
the closure of $\bigcup_{k \in \mathbb N} \Delta_k / k$ in $\mathcal M \otimes \mathbb R$, 
where $\Delta_k$ is a finite set consisting of (dominant) weights $\lambda$ such that 
\begin{equation*}
H^0(X, L^{\otimes k}) = \bigoplus_{\lambda \in \Delta_k} V_G(\lambda).
\end{equation*} 
Here, $V_G(\lambda)$ means the irreducible representation of $G$ with highest weight $\lambda$. 
\end{definition}

For a compact connected Lie group $K$ and a compact connected Hamiltonian $K$-manifold $(M, \omega, \mu)$, 
Kirwan \cite{Kirwan84} proved that the intersection of the image of $M$ through the moment map $\mu$ 
with the positive Weyl chamber with respect to a Borel subgroup $B$ of $G$ is a convex polytope, 
where $G$ is the complexification of $K$. 
The algebraic moment polytope $\Delta(X, L)$ for a polarized $G$-variety $X$ was introduced by Brion in \cite{Brion87} 
as a purely algebraic version of the Kirwan polytope. 
This is indeed the convex hull of finitely many points in $\mathcal M \otimes \mathbb R$ (see \cite{Brion87}). 
Moreover, if $X$ is smooth, then $\Delta(X, L)$ can be interpreted as 
the Kirwan polytope of $(X, \omega_L)$ with respect to the action of a maximal compact subgroup $K$ of $G$, 
where $\omega_L$ is a $K$-invariant K\"{a}hler form in the first Chern class $c_1(L)$. 

\begin{example}[Equivariant compactifications of reductive groups]
Any reductive group $G$ is spherical with respect to the action of $G \times G$ by left and right multiplication from the Bruhat decomposition. 
Let us consider the \emph{wonderful compactification} of a simple algebraic group $G$ of adjoint type constructed by De~Concini and Procesi \cite{dCP83}.  
As a specific example, the wonderful compactification $\mathbb P(\text{Mat}_{2 \times 2}(\mathbb C)) \cong \mathbb P^3$ 
of the projective general linear group $\PGL(2, \mathbb C)$ has 
the action of $\PGL(2, \mathbb C) \times \PGL(2, \mathbb C)$ induced by the multiplication of matrices on the left and on the right. 
It is known that the spherical weight lattice $\mathcal M$ of the wonderful compactification of a simple algebraic group $G$ of adjoint type 
coincides with the root lattice of $G$. 
As the anticanonical line bundle $K_{\mathbb P^3}^{-1}$ is isomorphic to $\mathcal O_{\mathbb P^3}(4)$, 
\begin{align*}
H^0(\mathbb P^3, K_{\mathbb P^3}^{-1}) & = \Sym^4 \mathbb C^4 
\\
 & \cong \End(V_{\PGL(2, \mathbb C)}(0)) \oplus \End(V_{\PGL(2, \mathbb C)}(2\varpi_1)) \oplus \End(V_{\PGL(2, \mathbb C)}(4\varpi_1)),
\end{align*} 
where $\varpi_1$ denotes the fundamental weight of $\PGL(2, \mathbb C)$. 
Repeating this calculation for tensor powers $(K_{\mathbb P^3}^{-1})^{\otimes k}$, we obtain 
$$\frac{1}{k} \Delta_k = \left\{ 0, \frac{2}{k}\varpi_1, \frac{4}{k}\varpi_1, \cdots , \frac{4k-2}{k}\varpi_1, \frac{4k}{k}\varpi_1 \right\}.$$ 
Therefore, the algebraic moment polytope $\Delta(\mathbb P^3, K_{\mathbb P^3}^{-1})$ of the wonderful compactification of $\PGL(2, \mathbb C)$ is 
a closed interval $[0, 4 \varpi_1] = [0, 2 \alpha_1]$ in $\mathcal M \otimes \mathbb R \cong \mathbb R \cdot \alpha_1$, 
where $\alpha_1$ denotes the simple root of $\PGL(2, \mathbb C)$. 
\end{example}

\subsection{Symmetric spaces and symmetric varieties}

For an algebraic group involution $\theta$ of a connected reductive algebraic group $G$, 
let 
\begin{equation*}
G^{\theta}=\{ g \in G : \theta(g)=g \}
\end{equation*} 
be the subgroup consisting of elements fixed by $\theta$. 
If $H$ is a closed subgroup of $G$ such that the identity component of $H$ coincides with the identity component of $G^{\theta}$, 
then the homogeneous space $G/H$ is called a \emph{symmetric homogeneous space}. 
By taking a universal cover of $G$, we can always assume that $G$ is simply connected. 
When $G$ is simply connected, 
$G^{\theta}$ is connected (see Section 8.1 in \cite{Seinberg68}) and $H$ is a closed subgroup between $G^{\theta}$ and its normalizer $N_G(G^{\theta})$ in $G$, that is, $G^{\theta} \subset H \subset N_G(G^{\theta})$. 

\begin{definition}
A normal $G$-variety $X$ together with an equivariant open embedding $G/H \hookrightarrow X$ 
of a symmetric homogeneous space $G/H$ is called a \emph{symmetric variety}. 
\end{definition}

Vust proved that a symmetric homogeneous space $G/H$ is spherical (see \cite[Theorem 1 in Section 1.3]{Vust73}). 
By using the Luna--Vust theory on spherical varieties, 
Ruzzi \cite{Ruzzi2011} classified the smooth projective symmetric varieties with Picard number one 
from the classification of corresponding colored fans.
As a result, there are only 6 nonhomogeneous smooth projective symmetric varieties with Picard number one, 
and their restricted root systems (see Subsection \ref{criterion for symmetric varieties} for the definition) are of either type $A_2$ or type $G_2$. 
Moreover, Ruzzi gave geometric descriptions of them in \cite{Ruzzi2010}.

In the case that the restricted root system is of type $A_2$ (Theorem 3 of \cite{Ruzzi2010}), 
the symmetric varieties are the smooth equivariant completions of symmetric homogeneous spaces 
$\SL(3, \mathbb C)/\SO(3, \mathbb C)$, 
$(\SL(3, \mathbb C) \times \SL(3, \mathbb C))/\SL(3, \mathbb C)$, 
$\SL(6, \mathbb C)/\Sp(6, \mathbb C)$, $E_6/F_4$, 
and are isomorphic to a general hyperplane section of rational homogeneous manifolds 
which are in the third row of the \emph{geometric Freudenthal--Tits magic square}. 
\begin{center}
\begin{tabular}{ c | c c c c }
\hline
   & $\mathbb R$ & $\mathbb C$ & $\mathbb H$ & $\mathbb O$  \\
 \hline
  $\mathbb R$ & $v_4(\mathbb P^1)$ & $\mathbb P (T_{\mathbb P^2})$ & $\Gr_{\omega}(2, 6)$ & $\mathbb{OP}^2_0$  \\
  $\mathbb C$ & $v_2(\mathbb P^2)$ & $\mathbb P^2 \times \mathbb P^2$ & $\Gr(2, 6)$ & $\mathbb{OP}^2$ \\
  $\mathbb H$ & $\LGr(3, 6)$ & $\Gr(3, 6)$ & $\mathbb S_6$ & $E_7/P_7$ \\
  $\mathbb O$ & $F_4^{ad}$ & $E_6^{ad}$ & $E_7^{ad}$ & $E_8^{ad}$ \\
\hline
\end{tabular}
\end{center}

\begin{remark}
Though all the rational homogeneous manifolds admit K\"{a}hler--Einstein metrics, 
a general hyperplane section of a rational homogeneous manifold is not necessarily the case. 
For example, a general hyperplane section of the Grassmannian $\Gr(2, 2n+1)$, called an \emph{odd symplectic Grassmannian of isotropic planes}, 
does not admit K\"{a}hler--Einstein metrics by the Matsushima theorem in \cite{matsushima} 
because the automorphism group of the odd symplectic Grassmannian is not reductive (see \cite{Mihai07}). 
\end{remark}

In the case that the restricted root system is of type $G_2$ (Theorem 2 of \cite{Ruzzi2010}), 
the symmetric varieties are the smooth equivariant completions of either $G_2/(\SL(2, \mathbb C) \times \SL(2, \mathbb C))$ or $(G_2 \times G_2)/G_2$.
The smooth equivariant completion with Picard number one of the symmetric space $G_2/(\SL(2, \mathbb C) \times \SL(2, \mathbb C))$, 
called the \emph{Cayley Grassmannian}, 
and the smooth equivariant completion with Picard number one of the symmetric space $(G_2 \times G_2)/G_2$, 
called the \emph{double Cayley Grassmannian}, 
have been studied by Manivel \cite{Manivel18, Manivel20}. 

Their geometric properties including the dimension, the Fano index, the restricted root system are listed in Table \ref{table}. 
For the deformation rigidity properties of smooth projective symmetric varieties with Picard number one, see \cite{KP19}.

\begin{table}[H]
\begin{tabular}{c c c c c c c}
		\toprule
$X_i$ & $G/G^{\theta}$ & $\dim X_i$ & $\iota(X_i)$ & $\Phi_{\theta}$ & Multiplicity & Geometric Description \\
		\midrule 
$1$	&	$\SL(3, \mathbb C)/\SO(3, \mathbb C)$	&	$5$  &	$3$ 	&	$A_2$ 	& 	$1$  	& 	hyperplane section of $\LGr(3, 6)$ 
\\
$2$	&	$(\SL(3, \mathbb C) \times \SL(3, \mathbb C))/\SL(3, \mathbb C)$	&	$8$  &	$5$ 	&	$A_2$ 	& 	$2$   	& 	hyperplane section of $\Gr(3, 6)$ 
\\
$3$	&	$\SL(6, \mathbb C)/\Sp(6, \mathbb C)$	&	$14$  &	$9$ 	&	$A_2$ 	& 	$4$  	& 	hyperplane section of $\mathbb S_6$    
\\
$4$	&	$E_6/F_4$	&	$26$  &	$17$ 	&	$A_2$ 	& 	$8$  	& 	hyperplane section of $E_7/P_7$    
\\
$5$	&	$G_2/(\SL(2, \mathbb C) \times \SL(2, \mathbb C))$	&	$8$  &	$4$ 	&	$G_2$ 	& 	$1$ 	&	 Cayley Grassmannian   
\\
$6$	&	$(G_2 \times G_2)/G_2$	&	$14$  &	$7$ 	&	$G_2$ 	& 	$2$ 	&	 double Cayley Grassmannian
\\
		\bottomrule
\end{tabular}
\caption{Nonhomogeneous smooth projective symmetric varieties with Picard number one}
\label{table}
\end{table}

\subsection{Existence of K\"{a}hler--Einstein metrics on symmetric varieties} \label{criterion for symmetric varieties}

We recall Delcroix's criterion for K-stability of smooth Fano symmetric varieties in \cite{Del16}.  

For an algebraic group involution $\theta$ of $G$, 
a torus $T$ in $G$ is \emph{split} if $\theta(t)=t^{-1}$ for any $t \in T$. 
A torus $T$ is \emph{maximally split} if $T$ is a $\theta$-stable maximal torus in $G$ which contains a split torus $T_s$ of maximal dimension among split tori. 
Then, $\theta$ descends to an involution of $\mathfrak X(T)$ for a maximally split torus $T$, 
and the rank of a symmetric homogeneous space $G/H$ is equal to the dimension of a maximal split subtorus $T_s$ of $T$. 

Let $\Phi = \Phi(G, T)$ be the root system of $G$ with respect to a maximally split torus $T$. 
By Lemma 1.2 of \cite{dCP83}, 
we can take a set of positive roots $\Phi^+$ such that either $\theta(\alpha) = \alpha$ or $\theta(\alpha)$ is a negative root for all $\alpha \in \Phi^+$;  
then, we denote $2 \rho_{\theta} = \sum_{\alpha \in \Phi^+ \backslash \Phi^{\theta}} \alpha$, where $\Phi^{\theta} = \{ \alpha \in \Phi : \theta(\alpha) = \alpha \}$. 
The set $\Phi_{\theta} = \{ \alpha - \theta(\alpha) : \alpha \in \Phi \backslash \Phi^{\theta} \}$ is a (possibly non-reduced) root system, 
which is called the \emph{restricted root system}.  
Let $\mathcal C^+_{\theta}$ denote the cone generated by positive restricted roots in $\Phi_{\theta}^+ = \{ \alpha - \theta(\alpha) : \alpha \in \Phi^+ \backslash \Phi^{\theta} \}$.

\begin{proposition}[Corollary 5.9 of \cite{Del16}]  
\label{criterion}
Let $X$ be a smooth Fano embedding of a symmetric homogeneous space $G/H$ associated to an involution $\theta$ of $G$. 
Then $X$ admits a K\"{a}hler--Einstein metric if and only if 
the barycenter of the moment polytope $\Delta(X, K_X^{-1})$ with respect to the Duistermaat--Heckman measure 
$$\prod_{\alpha \in \Phi^+ \backslash \Phi^{\theta}} \kappa(\alpha, p) \, dp$$ 
is in the relative interior of the translated cone $2 \rho_{\theta} + \mathcal C^+_{\theta}$, 
where $\kappa$ denotes the Killing form on the Lie algebra $\mathfrak g$ of $G$.
\end{proposition}

In fact, this result is a direct consequence of a combinatorial criterion for the existence of a K\"{a}hler--Ricci soliton on smooth Fano spherical varieties obtained by Delcroix \cite[Theorem A]{Del16}. 
The proof consists of the existence of a special equivariant test configuration with horospherical central fiber 
and the explicit computation of the modified Futaki invariant on Fano horospherical varieties.

\section{Moment polytopes of smooth Fano symmetric varieties and their barycenters} 

We prove in this section our main result Theorem \ref{Main theorem}. 
The proof combines Proposition \ref{criterion} together with the following result allowing us to compute (algebraic) moment polytopes of Fano symmetric varieties. 

\begin{proposition}
\label{moment polytope}
Let $X$ be a smooth Fano embedding of a symmetric space $G/G^{\theta}$. 
Then, there exist integers $m_i$ such that a Weil divisor $-K_X = \sum_{i=1}^k m_i D_i + \sum_{j=1}^{\ell} E_j$ represents the anticanonical line bundle $K_X^{-1}$
for colors $D_i$ and $G$-stable divisors $E_j$ in $X$, 
and the moment polytope $\Delta(X, K_X^{-1})$ is $2\rho_{\theta} + Q_X^*$, 
where the polytope $Q_X$ is the convex hull of the set 
\begin{equation*}
\left\{ \frac{\rho(D_i)}{m_i} : i= 1, \cdots, k \right\} \cup \{ \hat{\rho}(E_j) 
: j= 1, \cdots , \ell \}
\end{equation*}  
in $\mathcal N \otimes \mathbb R$ and 
its dual polytope $Q_X^*$ is defined as 
$\{ m \in \mathcal M\otimes \mathbb R : \langle n, m \rangle \geq -1 \text{ for every } n \in Q_X \}$. 
\end{proposition}

This statement is a specialization of a result of Gagliardi and Hofscheier (\cite{GH15}, Section 9) in which they studied the anticanonical line bundle on a Gorenstein Fano spherical variety. 
It is based on the works of Brion \cite{Brion89, Brion97} on algebraic moment polytopes and anticanonical divisors of Fano spherical varieties. 
For the convenience of the reader, we provide a sketch of the proof.

\begin{proof}
Let us recall results about the anticanonical line bundle on a spherical variety from Sections 4.1 and 4.2 in \cite{Brion97}. 
For a spherical $G$-variety $X$, 
there exists a $B$-semi-invariant global section $s \in \Gamma(X, K_X^{-1})$ with $\text{div}(s) = \sum_{i=1}^k m_i D_i + \sum_{j=1}^{\ell} E_j$. 
Furthermore, the $B$-weight of this section $s$ is the sum of $\alpha \in \Phi$ such that $\mathfrak{g}_{-\alpha}$ does not stabilize the open $B$-orbit in $X$.  
Thus, when $X$ is a symmetric variety associated to an involution $\theta$ of $G$, 
the weight of $s$ is equal to $2 \rho_{\theta} = \sum_{\alpha \in \Phi^+ \backslash \Phi^{\theta}} \alpha$. 

For a Gorenstein Fano spherical variety $X$,
Brion obtained the relation between the moment polytope $\Delta(X, K_X^{-1})$ and a polytope $\Delta_{-K_X}$ associated to the anticanonical divisor in Proposition 3.3 of \cite{Brion89}. 
More precisely, if $X$ is a smooth Fano embedding of $G/G^{\theta}$ 
then the moment polytope $\Delta(X, K_X^{-1})$ is $2\rho_{\theta} + \Delta_{-K_X}$
and a polytope $\Delta_{-K_X}$ associated to the anticanonical (Cartier) divisor $-K_X$ is the dual polytope $Q_X^*$. 
\end{proof}

Let $\Phi = \Phi(G, T)$ be the root system of $G$ with respect to a maximally split torus $T$. 
Fix a set of positive roots $\Phi^+$ such that either $\theta(\alpha) = \alpha$ or $\theta(\alpha)$ is a negative root for all $\alpha \in \Phi^+$. 
We recall that the \emph{coroot} $\alpha^{\vee}$ of a root $\alpha \in \Phi$ is defined as the unique element in the Lie algebra $\mathfrak t$ of $T$ 
such that $\alpha(x)=\frac{2 \kappa(x, \alpha^{\vee})}{\kappa(\alpha^{\vee}, \alpha^{\vee})}$ for all $x \in \mathfrak t$. 
Given a set of simple roots $\{ \alpha_1, \alpha_2, \cdots , \alpha_r \} \subset \Phi$, 
we define the fundamental weights $\{ \varpi_1, \varpi_2, \cdots , \varpi_r \}$ dual to the coroots 
by requiring $\langle \alpha_i^{\vee}, \varpi_j \rangle = \delta_{i,j}$ for $i, j = 1, 2, \cdots, r = \dim T$.

\subsection{Smooth Fano embedding of $\SL(3, \mathbb C)/\SO(3, \mathbb C)$ with Picard number one}

Considering the involution $\theta$ of $\SL(n, \mathbb C)$ defined by sending $g$ to the inverse of its transpose $\theta(g)=(g^t)^{-1}$, 
which is usually called of Type AI, 
the subgroup fixed by $\theta$ is $\SO(n, \mathbb C)$. 
As $\theta(\alpha) = - \alpha$ for $\alpha \in \Phi = \Phi_{\SL_3}$, 
the set $\Phi^{\theta}$ is empty and the restricted root system $\Phi_{\theta}$ is the double $2 \Phi$ of the root system $\Phi$. 
The spherical weight lattice $\mathcal M = \mathfrak X(T/T\cap G^{\theta})$ is formed by $2 \lambda$ for weights $\lambda \in \mathfrak X(T)$. 
Thus, the dual lattice $\mathcal N$ is generated by half of the coroots $\frac{1}{2} \alpha_1^{\vee}, \frac{1}{2} \alpha_2^{\vee}$ 
from the relation $\langle \alpha_i^{\vee}, \varpi_j \rangle = \delta_{i,j}$. 
In general, Vust \cite{Vust90} proved that when $G$ is semisimple and simply connected, 
the spherical weight lattice $\mathcal M$ of the symmetric space $G/G^{\theta}$ is the lattice of restricted weights determined by the restricted root system, 
which implies that $\mathcal N$ is the lattice generated by \emph{restricted coroots} forming a root system dual to the restricted root system $\Phi_{\theta}$.  

Let $X_1$ be the smooth Fano embedding of $\SL(3, \mathbb C)/\SO(3, \mathbb C)$ with Picard number one. 
Using the description in \cite{Ruzzi2010}, 
we know that the two colors $D_1, D_2$ and the $G$-stable divisor $E$ in $X_1$
have the images $\rho(D_1) = \frac{1}{2} \alpha_1^{\vee}$, $\rho(D_2) = \frac{1}{2} \alpha_2^{\vee}$ and 
$\hat{\rho}(E) = -\frac{1}{2} \varpi_1^{\vee} - \frac{1}{2} \varpi_2^{\vee} = -\frac{1}{2} \alpha_1^{\vee}-\frac{1}{2} \alpha_2^{\vee}$ in $\mathcal N$, respectively.
Recall from Theorem 6 of \cite{Ruzzi2010} that the maximal colored cones of its colored fan are 
$(\Cone(\alpha_1^{\vee}, -\varpi_1^{\vee}-\varpi_2^{\vee}), \{D_1\})$ and $(\Cone(\alpha_2^{\vee}, -\varpi_1^{\vee}-\varpi_2^{\vee}), \{D_2\})$. 
Then, we have two relations $2D_1 - D_2 - E =0$ and $-D_1 + 2 D_2 - E =0$, so that $D_1 = D_2 = E$ in the Picard group  $\Pic(X_1)$.

\begin{proposition} 
\label{moment polytope_1}
Let $X_1$ be the smooth Fano embedding of $\SL(3, \mathbb C)/\SO(3, \mathbb C)$ with Picard number one. 
The moment polytope $\Delta_1 = \Delta(X_1, K_{X_1}^{-1})$ is the convex hull of three points $0$, $6 \varpi_1$, $6\varpi_2$ in $\mathcal M \otimes \mathbb R$. 
\end{proposition}

\begin{proof}
From the colored data of $\SL(3, \mathbb C)/\SO(3, \mathbb C)$ and the $G$-orbit structure of $X_1$, 
we know the relation $-K_{X_1} = D_1 + D_2 + E$ of the anticanonical divisor. 
Using Proposition \ref{moment polytope}, 
$\rho(D_1)$, $\rho(D_2)$ and $\hat{\rho}(E)$ are used as inward-pointing facet normal vectors of the moment polytope $\Delta(X_1, K_{X_1}^{-1})$. 
First, $\rho(D_1) = \frac{1}{2} \alpha_1^{\vee}$ gives an inequality 
$$\left\langle \frac{1}{2} \alpha_1^{\vee}, x \cdot 2\varpi_1 + y \cdot 2\varpi_2 - 2 \rho_{\theta} \right\rangle = x -1 \geq -1$$ 
because $2 \rho_{\theta} = 2\alpha_1 + 2\alpha_2 = 2\varpi_1 + 2\varpi_2$. 
Similarly, as $\rho(D_2) = \frac{1}{2} \alpha_2^{\vee}$ gives a domain $\{ x \cdot 2\varpi_1 + y \cdot 2\varpi_2 \in \mathcal M \otimes \mathbb R : y \geq 0 \}$, 
the images of two colors $D_1, D_2$ determine the positive Weyl chamber. 
Lastly, $\hat{\rho}(E) = - \frac{1}{2} \alpha_1^{\vee} - \frac{1}{2} \alpha_2^{\vee}$ gives a domain $\{ x \cdot 2\varpi_1 + y \cdot 2\varpi_2 \in \mathcal M \otimes \mathbb R : x+y \leq 3 \}$. 
Thus the moment polytope $\Delta(X_1, K_{X_1}^{-1})$ is the intersection of three half-spaces, 
so that it is the convex hull of three points $0$, $6 \varpi_1$, $6\varpi_2$ in $\mathcal M \otimes \mathbb R$. 
\end{proof}

\begin{corollary} 
\label{barycenter_X1} 
The smooth Fano embedding $X_1$ of $\SL(3, \mathbb C)/\SO(3, \mathbb C)$ with Picard number one admits a K\"{a}hler--Einstein metric. 
\end{corollary}

\begin{proof}
Choosing a realization of the root system $A_2$ in the Euclidean plane $\mathbb R^2$ with  $\alpha_1 = (1, 0)$ and $\alpha_2 = \left(- \frac{1}{2}, \frac{\sqrt{3}}{2} \right)$, 
for $p=(x, y)$ 
we obtain its Duistermaat--Heckman measure 
\[
\prod_{\alpha \in \Phi^+} \kappa(\alpha, p) \, dp = x \Big(-\frac{x}{2}+\frac{\sqrt{3}}{2}y\Big) \Big(\frac{x}{2}+\frac{\sqrt{3}}{2}y\Big) \, dxdy.
\]
From Proposition \ref{moment polytope_1}, we can compute the volume 
\begin{align*}
\text{Vol}_{DH}(\Delta_1) &= 
\displaystyle \int_{0}^{\sqrt{3}} \int_{0}^{\sqrt{3}y} 
x \Big(-\frac{x}{2}+\frac{\sqrt{3}}{2}y\Big) \Big(\frac{x}{2}+\frac{\sqrt{3}}{2}y\Big) \, dxdy
\\
& + \displaystyle \int_{\sqrt{3}}^{2\sqrt{3}} \int_{0}^{6-\sqrt{3}y} 
x \Big(-\frac{x}{2}+\frac{\sqrt{3}}{2}y\Big) \Big(\frac{x}{2}+\frac{\sqrt{3}}{2}y\Big)  \, dxdy
= \frac{27}{5} \sqrt{3}
\end{align*}
and the barycenter 
$$\textbf{bar}_{DH}(\Delta_1) = (\bar{x}, \bar{y} ) = \frac{1 }{\text{Vol}_{DH}(\Delta_1)} \displaystyle \int_{\Delta_1}p \prod_{\alpha \in \Phi^+} \kappa(\alpha, p) \, dp  
= \Big(\frac{5}{4}, \frac{5 \sqrt{3}}{4}\Big) = \frac{5}{4} \times 2\rho_{\theta}$$ 
of the moment polytope $\Delta_1$ with respect to the Duistermaat--Heckman measure. 
Therefore, $\textbf{bar}_{DH}(\Delta_1)$ is in the relative interior of the translated cone $2 \rho_{\theta} + \mathcal C^+_{\theta}$ (see Figure \ref{Delta_1}),  
so $X_1$ admits a K\"{a}hler--Einstein metric by Proposition \ref{criterion}.
\end{proof}

\begin{figure}
 \begin{minipage}[b]{.45 \textwidth}
 \centering
\begin{tikzpicture}
\clip (-0.5,-1.3) rectangle (7,6); 
\begin{scope}[y=(60:1)]

\coordinate (pi1) at (1,0);
\coordinate (pi2) at (0,1);
\coordinate (v1) at ($6*(pi1)$);
\coordinate (v2) at ($6*(pi2)$);
\coordinate (a1) at (2,-1);
\coordinate (a2) at (-1,2);
\coordinate (barycenter) at (5/2,5/2);

\coordinate (Origin) at (0,0);
\coordinate (asum) at ($(a1)+(a2)$);
\coordinate (2rho) at ($2*(asum)$);

\foreach \x  in {-8,-7,...,9}{
  \draw[help lines,dashed]
    (\x,-8) -- (\x,9)
    (-8,\x) -- (9,\x) 
    [rotate=60] (\x,-8) -- (\x,9) ;
}

\fill (Origin) circle (2pt) node[below left] {0};
\fill (pi1) circle (2pt) node[below] {$\varpi_1$};
\fill (pi2) circle (2pt) node[above] {$\varpi_2$};
\fill (a1) circle (2pt) node[below] {$\alpha_1$};
\fill (a2) circle (2pt) node[above] {$\alpha_2$};
\fill (asum) circle (2pt) node[below right] {$\alpha_1+\alpha_2$};
\fill (2rho) circle (2pt) node[below] {$2\rho_{\theta}$};

\fill (v1) circle (2pt) node[below] {$6\varpi_1$};
\fill (v2) circle (2pt) node[above left] {$6\varpi_2$};

\fill (barycenter) circle (2pt) node[below right] {$\textbf{bar}_{DH}(\Delta_1)$};

\draw[->,,thick](Origin)--(pi1);
\draw[->,,thick](Origin)--(pi2); 
\draw[->,,thick](Origin)--(a1);
\draw[->,,thick](Origin)--(a2);
\draw[->,,thick](Origin)--(asum); 

\draw[thick,gray](Origin)--(v1);
\draw[thick,gray](Origin)--(v2);
\draw[thick,gray](v1)--(v2);

\draw [shorten >=-4cm, red, thick, dashed] (2rho) to ($(2rho)+(a1)$);
\draw [shorten >=-4cm, red, thick, dashed] (2rho) to ($(2rho)+(a2)$);
\end{scope}
\end{tikzpicture} 
\caption{$\Delta_1=\Delta(X_1,K^{-1}_{X_1})$}
\label{Delta_1}

\end{minipage}
\begin{minipage}[b]{.45 \textwidth}
 \centering
\begin{tikzpicture}
\clip (-0.5,-1.3) rectangle (5.5,6); 
\begin{scope}[y=(60:1)]

\coordinate (v1) at ($5*(pi1)$);
\coordinate (v2) at ($5*(pi2)$);
\coordinate (barycenter) at (20/9,20/9);

\foreach \x  in {-8,-7,...,9}{
  \draw[help lines,dashed]
    (\x,-8) -- (\x,9)
    (-8,\x) -- (9,\x) 
    [rotate=60] (\x,-8) -- (\x,9) ;
}

\fill (Origin) circle (2pt) node[below left] {0};
\fill (pi1) circle (2pt) node[below] {$\varpi_1$};
\fill (pi2) circle (2pt) node[above] {$\varpi_2$};
\fill (a1) circle (2pt) node[below] {$\alpha_1$};
\fill (a2) circle (2pt) node[above] {$\alpha_2$};
\fill (asum) circle (2pt) node[below right] {$\alpha_1+\alpha_2$};
\fill (2rho) circle (2pt) node[below] {$2\rho_{\theta}$};

\fill (v1) circle (2pt) node[below] {$5\varpi_1$};
\fill (v2) circle (2pt) node[above] {$5\varpi_2$};

\fill (barycenter) circle (2pt) node[below right] {$\textbf{bar}_{DH}(\Delta_2)$};

\draw[->,,thick](Origin)--(pi1);
\draw[->,,thick](Origin)--(pi2); 
\draw[->,,thick](Origin)--(a1);
\draw[->,,thick](Origin)--(a2);
\draw[->,,thick](Origin)--(asum); 

\draw[thick,gray](Origin)--(v1);
\draw[thick,gray](Origin)--(v2);
\draw[thick,gray](v1)--(v2);

\draw [shorten >=-4cm, red, thick, dashed] (2rho) to ($(2rho)+(a1)$);
\draw [shorten >=-4cm, red, thick, dashed] (2rho) to ($(2rho)+(a2)$);
\end{scope}
\end{tikzpicture} 
\caption{$\Delta_2=\Delta(X_2,K^{-1}_{X_2})$}
\label{Delta_2}
\end{minipage}

\end{figure}

\subsection{Smooth Fano embedding of $(\SL(3, \mathbb C) \times \SL(3, \mathbb C))/\SL(3, \mathbb C)$ with Picard number one}
\label{X_2}

Any reductive algebraic group $L$ is a symmetric homogeneous space $(L \times L) / \text{diag} (L)$ 
under the action of the group $G = L \times L$ for the involution $\theta(g_1, g_2)=(g_2, g_1)$, $g_1, g_2 \in L$. 
If $T$ is a maximal torus of $L$, then $T \times T$ is a maximal torus of $G$ and 
we get the spherical weight lattice 
$$\mathcal M = \mathfrak X((T \times T) / \text{diag} (T)) = \{ (\lambda, -\lambda) : \lambda \in \mathfrak X(T) \}.$$ 
Thus, $\mathcal M$ can be identified with $\mathfrak X(T)$ by the projection to the first coordinate. 
Under this identification,
the restricted root system $\Phi_{\theta}$ is identified with the root system $\Phi_L$ of $L$ with respect to $T$, 
and the dual lattice $\mathcal N$ is generated by the coroots $\alpha_1^{\vee}, \alpha_2^{\vee}, \cdots , \alpha_r^{\vee}$, where $r = \dim T$.   

Let $X_2$ be the smooth Fano embedding of $(\SL(3, \mathbb C) \times \SL(3, \mathbb C))/\SL(3, \mathbb C)$ with Picard number one. 
Using the description in \cite{Ruzzi2010}, 
we know that the two colors $D_1, D_2$ and the $G$-stable divisor $E$ in $X_2$
have the images $\rho(D_1) = \alpha_1^{\vee}$, $\rho(D_2) = \alpha_2^{\vee}$ 
and $\hat{\rho}(E) = -\alpha_1^{\vee} - \alpha_2^{\vee}$ in $\mathcal N$, respectively.

\begin{proposition} 
\label{moment polytope_2}
Let $X_2$ be the smooth Fano symmetric embedding of $(\SL(3, \mathbb C) \times \SL(3, \mathbb C))/\SL(3, \mathbb C)$ with Picard number one. 
The moment polytope $\Delta_2 = \Delta(X_2, K_{X_2}^{-1})$ is the convex hull of three points $0$, $5 \varpi_1$, $5\varpi_2$ in $\mathcal M \otimes \mathbb R$. 
\end{proposition}

\begin{proof}
From the colored data of $(\SL(3, \mathbb C) \times \SL(3, \mathbb C))/\SL(3, \mathbb C)$ and the $G$-orbit structure of $X_2$, 
we know the relation $-K_{X_2} = 2 D_1 + 2 D_2 + E$ of the anticanonical divisor. 
Using Proposition \ref{moment polytope}, 
$\frac{1}{2} \rho(D_1)$, $\frac{1}{2} \rho(D_2)$ and $\hat{\rho}(E)$ are used as inward-pointing facet normal vectors of the moment polytope $\Delta(X_2, K_{X_2}^{-1})$. 
First, $\frac{1}{2} \rho(D_1) = \frac{1}{2} \alpha_1^{\vee}$ gives an inequality 
$$ \left\langle \frac{1}{2} \alpha_1^{\vee}, x \cdot \varpi_1 + y \cdot \varpi_2 - 2 \rho_{\theta} \right\rangle = \frac{1}{2} (x - 2) \geq -1$$  
because $2 \rho_{\theta} = 2\alpha_1 + 2\alpha_2 = 2\varpi_1 + 2\varpi_2$. 
Similarly, as $\frac{1}{2} \rho(D_2) = \frac{1}{2} \alpha_2^{\vee}$ gives a domain $\{ x \cdot \varpi_1 + y \cdot \varpi_2 \in \mathcal M \otimes \mathbb R : y \geq 0 \}$, 
the images of two colors $D_1, D_2$ determine the positive Weyl chamber. 
Lastly, $\hat{\rho}(E) = - \alpha_1^{\vee} - \alpha_2^{\vee}$ gives a domain $\{ x \cdot \varpi_1 + y \cdot \varpi_2 \in \mathcal M \otimes \mathbb R : x+y \leq 5 \}$. 
Thus, the moment polytope $\Delta(X_2, K_{X_2}^{-1})$ is the intersection of three half-spaces, 
so that it is the convex hull of three points $0$, $5 \varpi_1$, $5 \varpi_2$ in $\mathcal M \otimes \mathbb R$. 
\end{proof}

\begin{corollary} 
\label{barycenter_X2} 
The smooth Fano embedding $X_2$ of $(\SL(3, \mathbb C) \times \SL(3, \mathbb C))/\SL(3, \mathbb C)$ with Picard number one admits a K\"{a}hler--Einstein metric. 
\end{corollary}

\begin{proof}
As in the proof of Corollary \ref{barycenter_X1}, we choose a realization of the root system $A_2$ in the Euclidean plane $\mathbb R^2$ with  $\alpha_1 = (1, 0)$ and $\alpha_2 = \left(- \frac{1}{2}, \frac{\sqrt{3}}{2} \right)$. 
Then, the Duistermaat--Heckman measure on the moment polytope is given as 
\[
\prod_{\alpha \in \Phi^+} \kappa(\alpha, p) \, dp 
= \prod_{\beta \in \Phi^+_{\SL(3, \mathbb C)}} \kappa(\beta, p)^2 \, dp 
= x^2 \Big(-\frac{x}{2}+\frac{\sqrt{3}}{2}y\Big)^2 \Big(\frac{x}{2}+\frac{\sqrt{3}}{2}y\Big)^2 \, dxdy.
\]
From Proposition \ref{moment polytope_2}, we can compute the volume 
\begin{align*}
\text{Vol}_{DH}(\Delta_2) &= 
\displaystyle \int_{0}^{\frac{5}{6}\sqrt{3}} \int_{0}^{\sqrt{3}y} 
x^2 \Big(-\frac{x}{2}+\frac{\sqrt{3}}{2}y\Big)^2 \Big(\frac{x}{2}+\frac{\sqrt{3}}{2}y\Big)^2 \, dxdy
\\
&+ \displaystyle \int_{\frac{5}{6}\sqrt{3}}^{\frac{5}{3}\sqrt{3}} \int_{0}^{5-\sqrt{3}y} 
x^2 \Big(-\frac{x}{2}+\frac{\sqrt{3}}{2}y\Big)^2 \Big(\frac{x}{2}+\frac{\sqrt{3}}{2}y\Big)^2  \, dxdy
= \frac{78125}{18432} \sqrt{3}
\end{align*}
and the barycenter 
\begin{align*}
\textbf{bar}_{DH}(\Delta_2)  & = (\bar{x}, \bar{y} ) 
= \frac{1}{\text{Vol}_{DH}(\Delta_2)} \left( \displaystyle \int_{\Delta_2} x \prod_{\alpha \in \Phi^+} \kappa(\alpha, p) \, dp , \displaystyle \int_{\Delta_2} y \prod_{\alpha \in \Phi^+} \kappa(\alpha, p) \, dp \right)
\\
& = \left(\frac{10}{9}, \frac{10 \sqrt{3}}{9} \right) = \frac{10}{9} \times 2\rho_{\theta}
\end{align*} 
of the moment polytope $\Delta_2$ with respect to the Duistermaat--Heckman measure. 
Therefore, $\textbf{bar}_{DH}(\Delta_2)$ is in the relative interior of the translated cone $2 \rho_{\theta} + \mathcal C^+_{\theta}$ (see Figure \ref{Delta_2}), 
so $X_2$ admits a K\"{a}hler--Einstein metric by Proposition \ref{criterion}.
\end{proof}

\subsection{Smooth Fano embedding of $\SL(6, \mathbb C)/\Sp(6, \mathbb C)$ with Picard number one}

Recall the involution of Type AII. 
Let $\theta$ be an involution of $\SL(2m, \mathbb C)$ defined by $\theta(g) = J_m (g^t)^{-1} J_m^t$, 
where $J_m$ is the $2m \times 2m$ block diagonal matrix formed by 
$\big(\begin{smallmatrix}
0 & 1\\
-1 & 0
\end{smallmatrix}\big)$. 
Then, $G^{\theta} = \Sp(2m, \mathbb C)$ is the group of elements that preserve a nondegenerate skew-symmetric bilinear form $\omega(v, w)=v^t J_m w$. 
For $m=3$, we can check that the restricted root system $\Phi_{\theta}$ is the root system of type $A_2$ with multiplicity four, 
and the spherical weight lattice $\mathcal M = \mathfrak X(T/T\cap G^{\theta})$ 
is generated by $2 \lambda$ for weights $\lambda \in \mathfrak X(T_s)$,  
where $T_s$ denotes a split subtorus of dimension two in a maximal torus $T \subset \SL(6, \mathbb C)$.  
In fact, if we choose the torus of diagonal matrices as $T$, then the maximal split torus $T_s$ consists of diagonal matrices of the form 
$\mbox{diag}(a_1, a_1, a_2, a_2, a_3, a_3)$ with $a_1, a_2, a_3 \in \mathbb C^*$ and $a_1^2 a_2^2 a_3^2 = 1$. 
Denoting by $\alpha_k \colon T_s \to \mathbb C^*$ for $k=1, 2$ the characters defined by 
$$\alpha_k(\mbox{diag}(a_1, a_1, a_2, a_2, a_3, a_3))=\frac{a_k}{a_{k+1}}, $$
we have the restricted root system $\Phi_{\theta} = \{ \pm 2\alpha_1, \pm 2\alpha_2, \pm (2\alpha_1 + 2\alpha_2) \}$ of type $A_2$. 
Then, the dual lattice $\mathcal N$ is generated by the coroots $\frac{1}{2} \alpha_1^{\vee}, \frac{1}{2} \alpha_2^{\vee}$. 

Let $X_3$ be the smooth Fano embedding of $\SL(6, \mathbb C)/\Sp(6, \mathbb C)$ with Picard number one. 
Using the description in \cite{Ruzzi2010}, 
we know that the two colors $D_1, D_2$ and the $G$-stable divisor $E$ in $X_3$
have the images $\rho(D_1) = \frac{1}{2} \alpha_1^{\vee}$, $\rho(D_2) = \frac{1}{2} \alpha_2^{\vee}$ 
and $\hat{\rho}(E) = -\frac{1}{2} \alpha_1^{\vee} - \frac{1}{2} \alpha_2^{\vee}$ in $\mathcal N$, respectively.

\begin{proposition} 
Let $X_3$ be the smooth Fano symmetric embedding of $\SL(6, \mathbb C)/\Sp(6, \mathbb C)$ with Picard number one. 
The moment polytope $\Delta_3 = \Delta(X_3, K_{X_3}^{-1})$ is the convex hull of three points $0$, $18 \varpi_1$, $18 \varpi_2$ in $\mathcal M \otimes \mathbb R$. 
\end{proposition}

\begin{proof}
From the colored data of $\SL(6, \mathbb C)/\Sp(6, \mathbb C)$ and the $G$-orbit structure of $X_3$, 
we know the relation $-K_{X_3} = 4 D_1 + 4 D_2 + E$ of the anticanonical divisor. 
Using Proposition \ref{moment polytope}, 
$\frac{1}{4} \rho(D_1)$, $\frac{1}{4} \rho(D_2)$ and $\hat{\rho}(E)$ are used as inward-pointing facet normal vectors of the moment polytope $\Delta(X_3, K_{X_3}^{-1})$. 
Like the previous computations, $\frac{1}{4} \rho(D_1)$ and $\frac{1}{4} \rho(D_2)$ determine the positive restricted Weyl chamber. 
Indeed, $\frac{1}{4} \rho(D_1) = \frac{1}{8} \alpha_1^{\vee}$ gives an inequality 
$$\left\langle \frac{1}{8} \alpha_1^{\vee}, x \cdot 2 \varpi_1 + y \cdot 2 \varpi_2 - 2 \rho_{\theta} \right\rangle = \frac{1}{8} (2x - 8) \geq -1$$ because $2 \rho_{\theta} = 8\alpha_1 + 8\alpha_2 = 8\varpi_1 + 8\varpi_2$.
As $\hat{\rho}(E) = - \frac{1}{2} \alpha_1^{\vee} - \frac{1}{2} \alpha_2^{\vee}$ gives 
a domain $\{ x \cdot 2 \varpi_1 + y \cdot 2 \varpi_2 \in \mathcal M \otimes \mathbb R : x+y \leq 9 \}$, 
the moment polytope $\Delta(X_3, K_{X_3}^{-1})$ is the intersection of this half-space with the positive restricted Weyl chamber.  
Thus, $\Delta(X_3, K_{X_3}^{-1})$ is the convex hull of three points $0$, $18 \varpi_1$, $18 \varpi_2$ in $\mathcal M \otimes \mathbb R$.  
\end{proof}

\begin{corollary} 
\label{barycenter_X3}  
The smooth Fano embedding $X_3$ of $\SL(6, \mathbb C)/\Sp(6, \mathbb C)$ with Picard number one admits a K\"{a}hler--Einstein metric. 
\end{corollary}

\begin{proof}
As the multiplicity of each restricted root in the restricted root system $\Phi_{\theta}$ is four, 
the Duistermaat--Heckman measure on $\mathcal M \otimes \mathbb R$ is given as  
\[
\prod_{\alpha \in \Phi^+ \backslash \Phi^{\theta}} \kappa(\alpha, p) \, dp 
= x^4 \Big(-\frac{x}{2}+\frac{\sqrt{3}}{2}y\Big)^4 \Big(\frac{x}{2}+\frac{\sqrt{3}}{2}y\Big)^4 \, dxdy.
\]
Then, the barycenter $$\textbf{bar}_{DH}(\Delta_3) = (\bar{x}, \bar{y}) = \left(\frac{21}{5}, \frac{21 \sqrt{3}}{5} \right) = \frac{21}{20} \times 2\rho_{\theta}$$ 
is in the relative interior of the translated cone $2 \rho_{\theta} + \mathcal C^+_{\theta}$ (see Figure \ref{Delta_3}).  
Therefore, $X_3$ admits a K\"{a}hler--Einstein metric by Proposition \ref{criterion}.
\end{proof}

\begin{figure}
\begin{minipage}[b]{.45 \textwidth}
 \centering
\begin{tikzpicture}
\clip (-0.25,-0.5) rectangle (6.5,6); 
\begin{scope}[transform canvas={scale=0.35},y=(60:1)]

\coordinate (v1) at ($18*(pi1)$);
\coordinate (v2) at ($18*(pi2)$);
\coordinate (2rho) at ($8*(asum)$);
\coordinate (barycenter) at ($21/20*(2rho)$);

\foreach \x  in {-19,-18,...,25}{
  \draw[help lines,dashed]
    (\x,-19) -- (\x,25)
    (-19,\x) -- (25,\x) 
    [rotate=60] (\x,-19) -- (\x,25) ;
}

\fill (Origin) circle (5pt) node[below left] {\fontsize{25}{60}\selectfont0};
\fill (pi1) circle (2pt) node[below] {$\varpi_1$};
\fill (pi2) circle (2pt) node[above] {$\varpi_2$};
\fill (a1) circle (2pt) node[below] {$\alpha_1$};
\fill (a2) circle (2pt) node[above] {$\alpha_2$};
\fill (asum) circle (2pt) node[below right] {$\alpha_1+\alpha_2$};
\fill (2rho) circle (5pt) node[below left] {\fontsize{25}{60}\selectfont$2\rho_{\theta}$};

\fill (v1) circle (5pt) node[below left] {\fontsize{25}{60}\selectfont$18\varpi_1$};
\fill (v2) circle (5pt) node[above] {\fontsize{25}{60}\selectfont$18\varpi_2$};

\fill (barycenter) circle (5pt) node[above right] {\fontsize{25}{60}\selectfont$\textbf{bar}_{DH}(\Delta_3)$};

\draw[->,,thick](Origin)--(pi1);
\draw[->,,thick](Origin)--(pi2); 
\draw[->,,thick](Origin)--(a1);
\draw[->,,thick](Origin)--(a2);
\draw[->,,thick](Origin)--(asum); 

\draw[thick,gray](Origin)--(v1);
\draw[thick,gray](Origin)--(v2);
\draw[thick,gray](v1)--(v2);

\draw [shorten >=-10cm, red, ultra thick, dashed] (2rho) to ($(2rho)+(a1)$);
\draw [shorten >=-14cm, red, ultra thick, dashed] (2rho) to ($(2rho)+(a2)$);
\end{scope}
\end{tikzpicture} 
\caption{$\Delta_3=\Delta(X_3,K^{-1}_{X_3})$}
\label{Delta_3}
\end{minipage}
\begin{minipage}[b]{.45 \textwidth}
 \centering
 \begin{tikzpicture}
\clip (-0.2,-0.5) rectangle (6.5,6.25); 
\begin{scope}[transform canvas={scale=0.19},y=(60:1)]

\coordinate (v1) at ($34*(pi1)$);
\coordinate (v2) at ($34*(pi2)$);
\coordinate (2rho) at ($16*(asum)$);
\coordinate (barycenter) at ($221/216*(2rho)$);

\foreach \x  in {-35,-34,...,50}{
  \draw[help lines,dashed]
    (\x,-35) -- (\x,50)
    (-35,\x) -- (50,\x) 
    [rotate=60] (\x,-35) -- (\x,50) ;
}

\fill (Origin) circle (5pt) node[below left] {\fontsize{50}{60}\selectfont0};
\fill (pi1) circle (2pt) node[below] {$\varpi_1$};
\fill (pi2) circle (2pt) node[above] {$\varpi_2$};
\fill (a1) circle (2pt) node[below] {$\alpha_1$};
\fill (a2) circle (2pt) node[above] {$\alpha_2$};
\fill (asum) circle (2pt) node[below right] {$\alpha_1+\alpha_2$};
\fill (2rho) circle (5pt) node[below left] {\fontsize{50}{60}\selectfont $2\rho_{\theta}$};

\fill (v1) circle (5pt) node[below left] {\fontsize{50}{60}\selectfont $34\varpi_1$};
\fill (v2) circle (5pt) node[above] {\fontsize{50}{60}\selectfont $34\varpi_2$};

\fill (barycenter) circle (5pt) node[above right] {\fontsize{50}{60}\selectfont $\textbf{bar}_{DH}(\Delta_4)$};

\draw[->,,thick](Origin)--(pi1);
\draw[->,,thick](Origin)--(pi2); 
\draw[->,,thick](Origin)--(a1);
\draw[->,,thick](Origin)--(a2);
\draw[->,,thick](Origin)--(asum); 

\draw[thick,gray](Origin)--(v1);
\draw[thick,gray](Origin)--(v2);
\draw[thick,gray](v1)--(v2);

\draw [shorten >=-15cm, red, ultra thick, dashed] (2rho) to ($(2rho)+(a1)$);
\draw [shorten >=-20cm, red, ultra thick, dashed] (2rho) to ($(2rho)+(a2)$);
\end{scope}

\end{tikzpicture} 
\caption{$\Delta_4=\Delta(X_4,K^{-1}_{X_4})$}
\label{Delta_4}
\end{minipage}

\end{figure}

\subsection{Smooth Fano embedding of $E_6/F_4$ with Picard number one}

Let $\theta$ be the involution on the simple algebraic group $E_6$ of Type EIV. 
Then, $G^{\theta}$ is isomorphic to the simple algebraic group $F_4$, 
and the restricted root system $\Phi_{\theta}$ is the root system of type $A_2$ generated by the simple restricted roots $2 \alpha_1, 2 \alpha_2$ with multiplicity eight. 
The spherical weight lattice $\mathcal M = \mathfrak X(T/T\cap G^{\theta})$ 
is generated by $2 \lambda$ for weights $\lambda \in \mathfrak X(T_s)$, 
where $T_s$ denotes a split subtorus of dimension two in a maximal torus $T \subset E_6$,  
so that the dual lattice $\mathcal N$ is generated by the coroots $\frac{1}{2} \alpha_1^{\vee}, \frac{1}{2} \alpha_2^{\vee}$. 

Let $X_4$ be the smooth Fano embedding of $E_6/F_4$ with Picard number one. 
Using the description in \cite{Ruzzi2010}, 
we know that the two colors $D_1, D_2$ and the $G$-stable divisor $E$ in $X_4$
have the images $\rho(D_1) = \frac{1}{2} \alpha_1^{\vee}$, $\rho(D_2) = \frac{1}{2} \alpha_2^{\vee}$ 
and $\hat{\rho}(E) = -\frac{1}{2} \alpha_1^{\vee} - \frac{1}{2} \alpha_2^{\vee}$ in $\mathcal N$, respectively.

\begin{proposition} 
Let $X_4$ be the smooth Fano symmetric embedding of $E_6/F_4$ with Picard number one. 
The moment polytope $\Delta_4 = \Delta(X_4, K_{X_4}^{-1})$ is the convex hull of three points $0$, $34 \varpi_1$, $34 \varpi_2$ in $\mathcal M \otimes \mathbb R$. 
\end{proposition}

\begin{proof}
From the colored data of $E_6/F_4$ and the $G$-orbit structure of $X_4$, 
we know the relation $-K_{X_4} = 8 D_1 + 8 D_2 + E$ of the anticanonical divisor. 
Using Proposition \ref{moment polytope}, 
$\frac{1}{8} \rho(D_1)$, $\frac{1}{8} \rho(D_2)$ and $\hat{\rho}(E)$ are used as inward-pointing facet normal vectors of the moment polytope $\Delta(X_4, K_{X_4}^{-1})$. 
In particular, $\frac{1}{8} \rho(D_1)$ and $\frac{1}{8} \rho(D_2)$ determine the positive restricted Weyl chamber. 
Indeed, $\frac{1}{8} \rho(D_1) = \frac{1}{16} \alpha_1^{\vee}$ gives an inequality 
$$\left\langle \frac{1}{16} \alpha_1^{\vee}, x \cdot 2 \varpi_1 + y \cdot 2 \varpi_2 - 2 \rho_{\theta} \right\rangle = \frac{1}{16} (2x - 16) \geq -1$$ because $2 \rho_{\theta} = 16 \alpha_1 + 16 \alpha_2 = 16 \varpi_1 + 16 \varpi_2$.
As $\hat{\rho}(E) = - \frac{1}{2} \alpha_1^{\vee} - \frac{1}{2} \alpha_2^{\vee}$ gives 
a domain $\{ x \cdot 2 \varpi_1 + y \cdot 2 \varpi_2 \in \mathcal M \otimes \mathbb R : x+y \leq 17 \}$, 
the moment polytope $\Delta(X_4, K_{X_4}^{-1})$ is the intersection of this half-space with the positive restricted Weyl chamber.  
Thus, $\Delta(X_4, K_{X_4}^{-1})$ is the convex hull of three points $0$, $34 \varpi_1$, $34 \varpi_2$ in $\mathcal M \otimes \mathbb R$.  
\end{proof}

\begin{corollary} 
\label{barycenter_X4} 
The smooth Fano embedding $X_4$ of $E_6/F_4$ with Picard number one admits a K\"{a}hler--Einstein metric. 
\end{corollary}

\begin{proof}
As the multiplicity of each restricted root in the restricted root system $\Phi_{\theta}$ is eight, 
the Duistermaat--Heckman measure on $\mathcal M \otimes \mathbb R$ is given as 
\[
\prod_{\alpha \in \Phi^+ \backslash \Phi^{\theta}} \kappa(\alpha, p) \, dp 
= x^8 \Big(-\frac{x}{2}+\frac{\sqrt{3}}{2}y\Big)^8 \Big(\frac{x}{2}+\frac{\sqrt{3}}{2}y\Big)^8 \, dxdy.
\]
Then, the barycenter $$\textbf{bar}_{DH}(\Delta_4) = (\bar{x}, \bar{y}) = \left(\frac{221}{27}, \frac{221 \sqrt{3}}{27} \right) = \frac{221}{216} \times 2\rho_{\theta}$$ is in the relative interior of the translated cone $2 \rho_{\theta} + \mathcal C^+_{\theta}$ (see Figure \ref{Delta_4}).  
Therefore, $X_4$ admits a K\"{a}hler--Einstein metric by Proposition \ref{criterion}.
\end{proof}

\subsection{Smooth Fano embedding of $G_2/(\SL(2, \mathbb C) \times \SL(2, \mathbb C))$ with Picard number one}

Let $\theta$ be the unique nontrivial involution on the simple algebraic group $G_2$. 
Then, $G^{\theta}$ is isomorphic to $\SL(2, \mathbb C) \times \SL(2, \mathbb C)$, 
but $\Phi^{\theta}$ is empty and the restricted root system $\Phi_{\theta}$ is the root system of type $G_2$. 
The spherical weight lattice $\mathcal M = \mathfrak X(T/T\cap G^{\theta})$ 
is generated by $2 \lambda$ for weights $\lambda \in \mathfrak X(T)$ of a maximal torus $T \subset G_2$, 
so that the dual lattice $\mathcal N$ is generated by the coroots $\frac{1}{2} \alpha_1^{\vee}, \frac{1}{2} \alpha_2^{\vee}$. 

Let $X_5$ be the smooth Fano embedding of $G_2/(\SL(2, \mathbb C) \times \SL(2, \mathbb C))$ with Picard number one. 
Using the description in \cite{Ruzzi2010}, 
we know that the two colors $D_1, D_2$ and the $G$-stable divisor $E$ in $X_5$
have the images $\frac{1}{2} \alpha_1^{\vee}$, $\frac{1}{2} \alpha_2^{\vee}$ and $-\frac{1}{2} \varpi_2^{\vee}$ in $\mathcal N$, respectively.
Recall that the maximal colored cone of its colored fan is $(\Cone(\alpha_2^{\vee}, -\varpi_2^{\vee}), \{D_2\})$ from Theorem 6 of \cite{Ruzzi2010}. 
Then we have two relations $2D_1 - D_2 =0$ and $-3 D_1 + 2 D_2 - E =0$, so that $D_2 = 2 D_1 = 2 E$ in $\Pic(X_5)$.

Choose a realization of the root system $G_2$ in the Euclidean plane $\mathbb R^2$ with  $\alpha_1 = (1, 0)$ and $\alpha_2 = \left(- \frac{3}{2}, \frac{\sqrt{3}}{2} \right)$.
Then, the complex Lie group $G_2$ has 6 positive roots:
\begin{align*}
\Phi^+ &= \left\{ \alpha_1, \alpha_2, \alpha_1 + \alpha_2, 2 \alpha_1 + \alpha_2, 3\alpha_1 + \alpha_2, 3\alpha_1 + 2\alpha_2 \right\}
\\
& = \left\{ (1, 0), \left(-\frac{3}{2}, \frac{\sqrt{3}}{2}\right), \left(-\frac{1}{2}, \frac{\sqrt{3}}{2}\right), \left(\frac{1}{2}, \frac{\sqrt{3}}{2}\right), \left(\frac{3}{2}, \frac{\sqrt{3}}{2}\right), (0, \sqrt{3}) \right\}.
\end{align*}
From the relation $(\alpha_i^{\vee}, \varpi_j)=\delta_{ij}$, 
the fundamental weights corresponding to the system of simple roots are 
$\varpi_1 = \left(\frac{1}{2}, \frac{\sqrt{3}}{2} \right)$, $\varpi_2 = (0, \sqrt{3})$.

\begin{figure}
 \begin{minipage}[b]{.45 \textwidth}
 \centering

\begin{tikzpicture}
\clip (-1.7,-0.5) rectangle (4.5,7.5); 
\begin{scope}[y=(60:1)]

\coordinate (pi1) at (0,1);
\coordinate (pi2) at (-1,2);
\coordinate (v1) at ($8*(pi1)$);
\coordinate (v2) at ($4*(pi2)$);
\coordinate (a1) at (1,0);
\coordinate (a2) at (-2,1);
\coordinate (a3) at ($3*(a1)+(a2)$);
\coordinate (barycenter) at (512/273-32/9,64/9);

\coordinate (Origin) at (0,0);
\coordinate (asum) at ($(a1)+(a2)$);
\coordinate (2rho) at (-2,6);

\foreach \x  in {-8,-7,...,9}{
  \draw[help lines,dashed]
    (\x,-8) -- (\x,9)
    (-8,\x) -- (9,\x) 
    [rotate=60] (\x,-8) -- (\x,9) ;
}

\fill (Origin) circle (2pt) node[below left] {0};
\fill (pi1) circle (2pt) node[right] {$\varpi_1$};
\fill (pi2) circle (2pt) node[right] {$\varpi_2$};
\fill (a1) circle (2pt) node[below] {$\alpha_1$};
\fill (a2) circle (2pt) node[above] {$\alpha_2$};
\fill (a3) circle (2pt) node[below right] {$3\alpha_1+\alpha_2$};

\fill (asum) circle (2pt) node[above] {$\alpha_1+\alpha_2$};
\fill (2rho) circle (2pt) node[below] {$2\rho_{\theta}$};

\fill (v1) circle (2pt) node[above] {$8\varpi_1$};
\fill (v2) circle (2pt) node[above] {$4\varpi_2$};

\fill (barycenter) circle (2pt) node[below right] {$\textbf{bar}_{DH}(\Delta_5)$};

\draw[->,,thick](Origin)--(pi1);
\draw[->,,thick](Origin)--(pi2); 
\draw[->,,thick](Origin)--(a1);
\draw[->,,thick](Origin)--(a2);
\draw[->,,thick](Origin)--(a3); 
\draw[->,,thick](Origin)--(asum); 

\draw[thick,gray](Origin)--(v1);
\draw[thick,gray](Origin)--(v2);
\draw[thick,gray](v1)--(v2);

\draw [shorten >=-4cm, red, thick, dashed] (2rho) to ($(2rho)+(a1)$);
\draw [shorten >=-4cm, red, thick, dashed] (2rho) to ($(2rho)+(a2)$);
\end{scope}
\end{tikzpicture} 

\caption{$\Delta_5=\Delta(X_5,K^{-1}_{X_5})$}
\label{Delta_5}
\end{minipage}
\begin{minipage}[b]{.45 \textwidth}
 \centering

\begin{tikzpicture}
\clip (-1.7,-0.5) rectangle (4.5,7.5); 
\begin{scope}[y=(60:1)]

\coordinate (v1) at ($7*(pi1)$);
\coordinate (v2) at ($7/2*(pi2)$);
\coordinate (barycenter) at (1.759-49/15,98/15);

\foreach \x  in {-8,-7,...,9}{
  \draw[help lines,dashed]
    (\x,-8) -- (\x,9)
    (-8,\x) -- (9,\x) 
    [rotate=60] (\x,-8) -- (\x,9) ;
}

\fill (Origin) circle (2pt) node[below left] {0};
\fill (pi1) circle (2pt) node[right] {$\varpi_1$};
\fill (pi2) circle (2pt) node[right] {$\varpi_2$};
\fill (a1) circle (2pt) node[below] {$\alpha_1$};
\fill (a2) circle (2pt) node[above] {$\alpha_2$};
\fill (a3) circle (2pt) node[below right] {$3\alpha_1+\alpha_2$};

\fill (asum) circle (2pt) node[above] {$\alpha_1+\alpha_2$};
\fill (2rho) circle (2pt) node[below] {$2\rho_{\theta}$};

\fill (v1) circle (2pt) node[above] {$7\varpi_1$};
\fill (v2) circle (2pt) node[above] {$\frac{7}{2}\varpi_2$};

\fill (barycenter) circle (2pt) node[below right] {$\textbf{bar}_{DH}(\Delta_6)$};

\draw[->,,thick](Origin)--(pi1);
\draw[->,,thick](Origin)--(pi2); 
\draw[->,,thick](Origin)--(a1);
\draw[->,,thick](Origin)--(a2);
\draw[->,,thick](Origin)--(a3); 
\draw[->,,thick](Origin)--(asum); 

\draw[thick,gray](Origin)--(v1);
\draw[thick,gray](Origin)--(v2);
\draw[thick,gray](v1)--(v2);

\draw [shorten >=-4cm, red, thick, dashed] (2rho) to ($(2rho)+(a1)$);
\draw [shorten >=-4cm, red, thick, dashed] (2rho) to ($(2rho)+(a2)$);
\end{scope}
\end{tikzpicture} 
\caption{$\Delta_6=\Delta(X_6,K^{-1}_{X_6})$}
\label{Delta_6}
\end{minipage}
\end{figure}

\begin{proposition} 
\label{moment polytope_5}
Let $X_5$ be the smooth Fano symmetric embedding of $G_2/(\SL(2, \mathbb C) \times \SL(2, \mathbb C))$ with Picard number one. 
The moment polytope $\Delta_5 = \Delta(X_5, K_{X_5}^{-1})$ is the convex hull of three points 
$0$, $8 \varpi_1$, $4\varpi_2$ in $\mathcal M \otimes \mathbb R$. 
\end{proposition}

\begin{proof}
From the colored data of $G_2/(\SL(2, \mathbb C) \times \SL(2, \mathbb C))$ and the $G$-orbit structure of $X_5$, 
we know the relation $-K_{X_5} = D_1 + D_2 + E$ of the anticanonical divisor. 
Using Proposition \ref{moment polytope}, 
$\rho(D_1)$, $\rho(D_2)$ and $\hat{\rho}(E)$ are used as inward-pointing facet normal vectors of the moment polytope $\Delta(X_5, K_{X_5}^{-1})$. 
As before, $\rho(D_1)$ and $\rho(D_2)$ determine the positive Weyl chamber. 
Indeed, $\rho(D_1) = \frac{1}{2} \alpha_1^{\vee}$ gives an inequality $$\left\langle \frac{1}{2} \alpha_1^{\vee}, x \cdot 2 \varpi_1 + y \cdot 2 \varpi_2 - 2 \rho_{\theta} \right\rangle = \frac{1}{2} (2x - 2) \geq -1$$ because $2 \rho_{\theta} = 10 \alpha_1 + 6\alpha_2 = 2\varpi_1 + 2\varpi_2$.
As $\hat{\rho}(E) = -\frac{1}{2} \varpi_2^{\vee} = \left(0, - \frac{1}{\sqrt{3}} \right)$ gives 
a domain $\{ x \cdot 2 \varpi_1 + y \cdot 2 \varpi_2 \in \mathcal M \otimes \mathbb R : x + 2y \leq 4 \}$ 
from $\langle \varpi_2^{\vee}, \varpi_1 \rangle = 1$ and $\langle \varpi_2^{\vee}, \varpi_2 \rangle = 2$, 
the moment polytope $\Delta(X_5, K_{X_5}^{-1})$ is the intersection of this half-space with the positive Weyl chamber.  
Thus, $\Delta(X_5, K_{X_5}^{-1})$ is the convex hull of three points 
$0$, $8 \varpi_1=(4, 4\sqrt{3})$, $4\varpi_2=(0, 4\sqrt{3})$ in $\mathcal M \otimes \mathbb R$. 
\end{proof}


\begin{corollary} 
\label{barycenter_X5} 
The smooth Fano embedding $X_5$ of $G_2/(\SL(2, \mathbb C) \times \SL(2, \mathbb C))$ with Picard number one admits a K\"{a}hler--Einstein metric. 
\end{corollary}

\begin{proof}
For $p=(x, y)$, the Duistermaat--Heckman measure on $\mathcal M \otimes \mathbb R$ is given as   
\[
\prod_{\alpha \in \Phi^+} \kappa(\alpha, p) \, dp 
= x \Big(-\frac{3}{2}x + \frac{\sqrt{3}}{2}y\Big) \Big(-\frac{1}{2}x + \frac{\sqrt{3}}{2}y\Big) \Big(\frac{1}{2}x + \frac{\sqrt{3}}{2}y\Big) \Big(\frac{3}{2}x + \frac{\sqrt{3}}{2}y\Big) (\sqrt{3}y) \, dxdy.
\]
From Proposition \ref{moment polytope_5}, we can compute the volume 
\begin{align*}
\text{Vol}_{DH}(\Delta_5) & = 
\displaystyle \int_{0}^{4\sqrt{3}} \int_{0}^{\frac{y}{\sqrt{3}}} 
x \Big(-\frac{3}{2}x + \frac{\sqrt{3}}{2}y\Big) \Big(-\frac{1}{2}x + \frac{\sqrt{3}}{2}y\Big) \Big(\frac{1}{2}x + \frac{\sqrt{3}}{2}y\Big) \Big(\frac{3}{2}x + \frac{\sqrt{3}}{2}y\Big) (\sqrt{3}y) \, dxdy 
\\
&= 29952 \sqrt{3},  
\end{align*}
and the barycenter $$\textbf{bar}_{DH}(\Delta_5) = (\bar{x}, \bar{y}) = \left(\frac{512}{273}, \frac{32 \sqrt{3}}{9} \right) \approx (1.875, 3.556 \times \sqrt{3})$$ 
of the moment polytope $\Delta_5$ with respect to the Duistermaat--Heckman measure. 
Therefore, $\textbf{bar}_{DH}(\Delta_5)$ is in the relative interior of the translated cone $2 \rho_{\theta} + \mathcal C^+_{\theta}$ (see Figure \ref{Delta_5}),  
so $X_5$ admits a K\"{a}hler--Einstein metric by Proposition \ref{criterion}.
\end{proof}

\subsection{Smooth Fano embedding of $(G_2 \times G_2)/G_2$ with Picard number one}

As explained in Subsection \ref{X_2}, 
the simple algebraic group $G_2$ can be considered as a symmetric homogeneous space $(G_2 \times G_2) / \text{diag} (G_2)$ 
under the action of the group $G_2 \times G_2$ for the involution $\theta(g_1, g_2)=(g_2, g_1)$, $g_1, g_2 \in G_2$. 
The spherical weight lattice $\mathcal M$ can be identified with the character group $\mathfrak X(T)$ of a maximal torus $T \subset G_2$ 
by the projection to the first coordinate, 
and the dual lattice $\mathcal N$ is generated by the coroots $\alpha_1^{\vee}, \alpha_2^{\vee}$.   

Let $X_6$ be the smooth Fano embedding of $(G_2 \times G_2)/G_2$ with Picard number one. 
Using the description in \cite{Ruzzi2010}, 
we know that the two colors $D_1, D_2$ and the $G$-stable divisor $E$ in $X_6$
have the images $\rho(D_1) = \alpha_1^{\vee}$, $\rho(D_2) = \alpha_2^{\vee}$ 
and $\hat{\rho}(E) = -\varpi_2^{\vee}$ in $\mathcal N$, respectively.


\begin{proposition} 
Let $X_6$ be the smooth Fano symmetric embedding of $(G_2 \times G_2)/G_2$ with Picard number one. 
The moment polytope $\Delta_6 = \Delta(X_6, K_{X_6}^{-1})$ is the convex hull of three points  
$0$, $7 \varpi_1$, $\frac{7}{2}\varpi_2$ in $\mathcal M \otimes \mathbb R$. 
\end{proposition}

\begin{proof}
From the colored data of $(G_2 \times G_2)/G_2$ and the $G$-orbit structure of $X_6$, 
we know the relation $-K_{X_6} = 2 D_1 + 2 D_2 + E$ of the anticanonical divisor. 
Using Proposition \ref{moment polytope}, 
$\frac{1}{2} \rho(D_1)$, $\frac{1}{2} \rho(D_2)$ and $\hat{\rho}(E)$ are used as inward-pointing facet normal vectors of the moment polytope $\Delta(X_6, K_{X_6}^{-1})$. 
As $2 \rho_{\theta} = 10 \alpha_1 + 6\alpha_2 = 2\varpi_1 + 2\varpi_2$, 
$\frac{1}{2} \rho(D_1) = \frac{1}{2} \alpha_1^{\vee}$ gives an inequality 
$$\left\langle \frac{1}{2} \alpha_1^{\vee}, x \cdot \varpi_1 + y \cdot \varpi_2 - 2 \rho_{\theta} \right\rangle = \frac{1}{2} (x - 2) \geq -1. $$
Therefore, $\frac{1}{2} \rho(D_1)$ and $\frac{1}{2} \rho(D_2)$ determine the positive Weyl chamber. 
In the same way, as $\hat{\rho}(E) = - \varpi_2^{\vee}$ gives 
a domain $\{ x \cdot \varpi_1 + y \cdot \varpi_2 \in \mathcal M \otimes \mathbb R : x + 2y \leq 7 \}$, 
the moment polytope $\Delta(X_6, K_{X_6}^{-1})$ is the intersection of this half-space with the positive Weyl chamber.  
Thus, $\Delta(X_6, K_{X_6}^{-1})$ is the convex hull of three points  
$0$, $7 \varpi_1= \left(\frac{7}{2}, \frac{7\sqrt{3}}{2} \right)$, $\frac{7}{2}\varpi_2= \left(0, \frac{7\sqrt{3}}{2} \right)$ in $\mathcal M \otimes \mathbb R$. 
\end{proof}

\begin{corollary} \label{barycenter_X6} 
The smooth Fano embedding $X_6$ of $(G_2 \times G_2)/G_2$ with Picard number one admits a K\"{a}hler--Einstein metric. 
\end{corollary}

\begin{proof}
We can compute the barycenter of $\Delta_6$ with respect to the Duistermaat--Heckman measure 
$$\textbf{bar}_{DH}(\Delta_6) = (\bar{x}, \bar{y}) = \left(\frac{139601}{79360}, \frac{49 \sqrt{3}}{15} \right) \approx (1.759, 3.267 \times \sqrt{3})$$ from 
$$\text{Vol}_{DH}(\Delta_6) = 
\displaystyle \int_{0}^{\frac{7\sqrt{3}}{2}} \int_{0}^{\frac{y}{\sqrt{3}}} 
x^2 \Big(-\frac{3}{2}x + \frac{\sqrt{3}}{2}y\Big)^2 \Big(-\frac{1}{2}x + \frac{\sqrt{3}}{2}y\Big)^2 \Big(\frac{1}{2}x + \frac{\sqrt{3}}{2}y\Big)^2 \Big(\frac{3}{2}x + \frac{\sqrt{3}}{2}y\Big)^2 (\sqrt{3}y)^2 \, dxdy .$$
As $2 \rho_{\theta} = (1, 3\sqrt{3})$ and the cone $\mathcal C^+_{\theta}$ is generated by the vectors $(1, 0)$ and $(-3, \sqrt{3})$, 
the barycenter $\textbf{bar}_{DH}(\Delta_6)$ is in the relative interior of the translated cone $2 \rho_{\theta} + \mathcal C^+_{\theta}$ (see Figure \ref{Delta_6}).   
Therefore, $X_6$ admits a K\"{a}hler--Einstein metric by Proposition \ref{criterion}.
\end{proof}

\vskip 1em 

Finally, by Ruzzi's classification of the smooth projective symmetric varieties with Picard number one in \cite{Ruzzi2011}, 
Corollaries \ref{barycenter_X1}, \ref{barycenter_X2}, \ref{barycenter_X3}, \ref{barycenter_X4}, \ref{barycenter_X5}, \ref{barycenter_X6} 
imply the following statement. Therefore, we conclude Theorem \ref{Main theorem}. 

\begin{theorem1}
All smooth Fano symmetric varieties with Picard number one admit \ke metrics. 
\end{theorem1}





\providecommand{\bysame}{\leavevmode\hbox to3em{\hrulefill}\thinspace}
\providecommand{\MR}{\relax\ifhmode\unskip\space\fi MR }
\providecommand{\MRhref}[2]{%
  \href{http://www.ams.org/mathscinet-getitem?mr=#1}{#2}
}
\providecommand{\href}[2]{#2}

\end{document}